\documentclass{amsart}

\usepackage{amsmath,amssymb,amsthm,amscd}

\usepackage[nohug,heads=LaTeX]{diagrams}

\newtheorem{prop}{Proposition}[section]
\newtheorem{thm}[prop]{Theorem}
\newtheorem{lem}[prop]{Lemma}
\newtheorem{cor}[prop]{Corollary}

\theoremstyle{definition}

\newtheorem{rem}[prop]{Remark}

\newtheorem*{ack}{Acknowledgements}

\def\co{\colon\thinspace}

\newcommand{\A}{\mathcal{A}}

\newcommand{\C}{\mathbb{C}}
\newcommand{\CP}{\mathbb{C}\mathrm{P}}

\newcommand{\Jst}{J_{\mathrm{st}}}

\newcommand{\OO}{\mathrm{O}}

\newcommand{\R}{\mathbb{R}}

\newcommand{\SO}{\mathrm{SO}}
\newcommand{\BSO}{\mathrm{BSO}}

\newcommand{\U}{\mathrm{U}}
\newcommand{\BU}{\mathrm{BU}}

\newcommand{\xist}{\xi_{\mathrm{st}}}

\newcommand{\Z}{\mathbb{Z}}

\DeclareMathOperator{\Int}{Int}

\DeclareMathOperator{\rank}{rank}

\begin{document}

\title{On subcritically Stein fillable $5$-manifolds}

\author[F.~Ding]{Fan Ding}
\address{School of Mathematical Sciences and LMAM, Peking University,
Beijing 100871, P.~R.~China}
\email{dingfan@math.pku.edu.cn}

\author[H.~Geiges]{Hansj\"{o}rg Geiges}
\address{Mathematisches Institut, Universit\"at zu K\"oln,
Weyertal 86--90, 50931 K\"oln, Germany}
\email{geiges@math.uni-koeln.de}

\author[G.~Zhang]{Guangjian Zhang}
\address{School of Mathematical Sciences, Peking University,
Beijing 100871, P.~R.~China}
\email{zhangguangjian8888@163.com}

\date{}

\begin{abstract}
We make some elementary observations concerning subcritically Stein
fillable contact structures on $5$-manifolds.
Specifically, we determine the diffeomorphism type of such
contact manifolds in the case the fundamental group is finite cyclic,
and we show that on the $5$-sphere the standard contact structure
is the unique subcritically fillable one. More generally,
it is shown that subcritically fillable contact structures
on simply connected $5$-manifolds are determined by their
underlying almost contact structure. Along the way, we discuss the
homotopy classification of almost contact structures.
\end{abstract}



\maketitle


\section{Introduction}
A \emph{Stein domain} in the sense of \cite[Definition 11.14]{ciel12}
is a compact manifold $W$ with boundary admitting a complex structure
$J$ and a $J$-convex Morse function for which the boundary
$\partial W=:M$ is a regular level set. We shall write a Stein domain
as a pair $(W,J)$ although, strictly speaking, the $J$-convex Morse function
is part of the data. The complex tangencies $TM\cap J(TM)$ define
a contact structure.

A closed contact manifold $(M,\xi)$ is said to be \emph{Stein fillable}
if it arises in this way as the boundary of a Stein domain.
It is well known that a Stein domain
of dimension $2n$ has a handle decomposition, adapted to
the Stein structure, with handles of index
at most equal to~$n$. A Stein filling is called \emph{subcritical},
if there are no handles of index~$n$.

In this note we are concerned with topological and contact geometric aspects
of subcritically Stein fillable contact $5$-manifolds.
The first result we want to discuss gives a uniqueness
statement for the diffeomorphism type of such contact manifolds
when it has a finite cyclic fundamental group.
This extends a corresponding result for simply connected contact manifolds
due to Bowden--Crowley--Stipsicz~\cite{bcs14}.
The main issue is one of simple homotopy theory, which in our examples can
be addressed with results of Hambleton--Kreck~\cite{hakr93} on $2$-complexes
and, as in~\cite{bcs14}, the Mazur--Wall theory of thickenings.

In order to state the result, we need to introduce certain
model manifolds. Let $m\geq 2$ be an integer. Write
$L_m$ for the $3$-dimensional lens space $L(m,1)$ with an open
$3$-disc removed. This space $L_m$ can be obtained from
a solid torus $S^1\times D^2$ by attaching a $2$-handle along
an $(m,-1)$-torus not in $\partial (S^1\times D^2)$.

Oriented $D^3$-bundles over $L_m$ are classified by the
second Stiefel--Whitney class $w_2$ (this standard fact will be
elucidated in the proof of Theorem~\ref{thm:topology}).
Since $H^2(L_m;\Z_2)$ is trivial for $m$ odd, and isomorphic to
$\Z_2$ for $m$ even, the only $D^3$-bundles are the product
$L_m\times D^3$ and, for $m=2n$ even, the non-trivial bundle
$L_{2n}\tilde\times D^3$. After rounding of corners, we may think
of the total spaces of these bundles as manifolds with boundary.

Similarly, over $S^2$ we have the trivial $S^3$-bundle $S^2\times S^3$,
and the non-trivial one $S^2\tilde\times S^3$. In
\cite[Proposition~7.4]{bcs14} it was shown that if $(M,\xi)$
is a closed, \emph{simply connected} $5$-dimensional contact
manifold admitting a subcritical Stein filling, then $M$ is
diffeomorphic to $\#_r S^2\times S^3$
if $M$ is spin, and $S^2\tilde\times S^3\#_{r-1}S^2\times S^3$
if $M$ is not spin, where $r=\rank H_2(M;\Z)$.

Our first result extends this to finite cyclic fundamental groups.

\begin{thm}
\label{thm:topology}
Suppose that $(M,\xi)$ is a closed, connected contact $5$-manifold admitting
a subcritical Stein filling, with $\pi_1(M)\cong \Z_m$ for some
integer $m\geq 2$. Set $r=\rank H_2(M;\Z)$.

\begin{itemize}
\item[(i)] If $m=2n+1$ is odd, then $M$ is diffeomorphic to
\[ \partial(L_{2n+1}\times D^3)\#_r(S^2\times S^3)
\;\;\text{or}\;\;
\partial(L_{2n+1}\times D^3)\# (S^2\tilde\times S^3)\#_{r-1}(S^2\times S^3),\]
depending on whether $M$ is spin or not.

\item[(ii)] If $m=2n$ is even, then $M$ is diffeomorphic to
\[ \partial(L_{2n}\times D^3)\#_r(S^2\times S^3)\]
if $M$ is spin, and to
\[ \partial(L_{2n}\tilde\times D^3)\#_r(S^2\times S^3)
\;\;\text{or}\;\;
\partial(L_{2n}\times D^3)\#(S^2\tilde\times S^3)\#_{r-1}(S^2\times S^3)\]
when $M$ is not spin.
\end{itemize}
The diffeomorphism type of the subcritical Stein filling is
determined by~$M$.

On any of these manifolds, each homotopy class of almost contact structures
contains a unique subcritically Stein fillable contact structure,
up to isotopy,
with a unique Stein filling, up to Stein homotopy.
\end{thm}

The strategy for proving this theorem is as follows.
First, we use homotopy-theoretic methods to arrive at
a topological classification of the potential subcritical fillings.
We then appeal to the fundamental work of Cieliebak--Eliashberg~\cite{ciel12}
that reduces the existence and classification question for Stein structures,
in the subcritical case, yet again to a problem of homotopy theory.
The relevant results from \cite{ciel12} will be recalled below.
That second homotopy-theoretic problem, the homotopy classification of
almost contact and almost complex structures on $5$-dimensional
and $6$-dimensional manifolds, respectively, is a matter of classical
obstruction theory; see Section~\ref{section:acs}.

The same strategy, combined with the results
of~\cite{bcs14}, allows us to complete our discussion
of subcritical Stein fillings of simply connected $5$-manifolds
in~\cite{dgk12}, which was written previous to \cite{ciel12}
being available.

\begin{thm}
\label{thm:pi1=1}
Let $M$ be a closed, simply connected $5$-manifold admitting
a subcritical Stein filling, that is, one of the manifolds
$\#_r S^2\times S^3$ or $S^2\tilde\times S^3\#_{r-1}S^2\times S^3$.
Then each homotopy class of almost contact structures
contains a unique subcritically Stein fillable contact structure,
up to isotopy,
with a unique Stein filling, up to Stein homotopy.
\end{thm}

Two particular consequences (or special cases) of this result
are worth noting.
Here, for $(M,\xi)$ a
contact manifold, $c_1(\xi)\in H^2(M;\Z)$ denotes the first
Chern class of~$\xi$; recall that a contact structure carries a
complex bundle structure, unique up to
homotopy~\cite[Proposition~2.4.8]{geig08}.

\begin{cor}
\label{cor:pi1=1}
(a) Any subcritically Stein fillable contact structure on $S^5$ is
isotopic to the standard contact structure.

(b) Let $(M_i,\xi_i)$, $i=1,2$, be two simply connected
subcritically Stein fillable contact $5$-manifolds. If there is an
isomorphism $\phi\co H^2(M_1;\Z)\rightarrow H^2(M_2;\Z)$ such that
$\phi(c_1(\xi_1))=c_1(\xi_2)$, then $(M_1,\xi_1)$ and $(M_2,\xi_2)$ are
contactomorphic.
\end{cor}

Part (a) confirms an expectation from
\cite[Section~6]{dgk12}. Part (b) had been proved
in \cite[Theorem 4.8]{dgk12} under the
additional assumption that the fillings contain no $1$-handles.
As pointed out by the referee, an isomorphism $\phi$
as required by part (b) of the corollary exists if and only
if the Chern classes $c_1(\xi_1)$ and $c_1(\xi_2)$ have the same
divisibility in the free abelian group $H^2(M_i;\Z)$,
see~\cite[Theorem~8.20]{hami08}. Thus, the divisibility of the
first Chern class is the only contactomorphism invariant
of a subcritically Stein fillable contact structure on
a simply connected $5$-manifold.
\section{Homotopy classification of almost contact structures}
\label{section:acs}
In this section we discuss the homotopy
classification of almost contact structures on $5$-manifolds,
correcting a negligence in~\cite{geig08}. Likewise, we describe
the classification of almost complex structures on $6$-manifolds,
correcting a similar oversight in~\cite{wall66b}. These
classification results are key ingredients in the proof
of Theorem~\ref{thm:topology} and Theorem~\ref{thm:pi1=1}.

A careful discussion of this homotopy classification can be found
in Mark Hamilton's thesis~\cite[VIII.4]{hami08}, and our reasoning
goes along the same lines. We show that by a closer
look at this obstruction-theoretic argument one can in fact exhibit
a free and transitive action of the second cohomology group on the
space of almost contact (resp.\ complex) structures.

Let $M$ be a compact (not necessarily closed), oriented $5$-manifold.
A choice of Riemannian metric
on~$M$, or equivalently a reduction of the structure group
of the tangent bundle to $\SO(5)$, allows us to describe the tangent
bundle $TM$ in terms of a classifying map $f\co M\rightarrow\BSO(5)$.
Define the inclusion $\U(2)\subset\SO(5)$ by the embedding
$\C^2\equiv\R^4\times\{0\}\subset\R^5$. Any subgroup $\mathrm{G}\subset\OO(5)$
acts on the space $V(5)$ of orthonormal $5$-frames in $\R^{\infty}$,
and this defines the universal bundles $V(5)\rightarrow\mathrm{BG}:=V(5)/
\mathrm{G}$. The quotient $\mathrm{BG}$ is the classifying space
for $\mathrm{G}$-bundles, see
\cite[Section~A.2]{whit78}. The inclusion $\U(2)\subset\SO(5)$
defines a fibration $p\co\BU(2)\rightarrow\BSO(5)$ with fibre
$F_5:=\SO(5)/\U(2)$.

An almost contact structure on $M$ is a reduction of the structure
group of $TM$ from $\SO(5)$ to $\U(2)$, which amounts to a lift
$\tilde{f}$ of the classifying map~$f$:
\begin{diagram}
          &                   & \BU(2)\\
          & \ruTo^{\tilde{f}} & \dTo_p      \\
M         & \rTo_{f}          & \BSO(5)
\end{diagram}

The lifting condition $p\circ\tilde{f}=f$ is equivalent to saying that
the map
\[ M\ni x\longmapsto \sigma(x):=\bigl(x,\tilde{f}(x)\bigr)
\in M\times V(5)/\U(2)\]
is a section of the induced bundle $E:=f^*V(5)/U(2)=f^*\BU(2)$ over $M$
with fibre~$F_5$. This is the obstruction-theoretic setting
of \cite[Part~III]{stee51}. Notice that $f^*V(5)$ is the frame
bundle of~$M$.

From now on we shall interpret almost contact structures on $M$ as
sections $\sigma$ of this bundle $E\rightarrow M$. Homotopy of almost
contact structures means homotopy of such sections.

\begin{lem}
\label{lem:section}
The $\U(2)$-bundle $\tilde{f}^*V(5)\rightarrow M$ corresponding to the
almost contact structure defined by $\tilde{f}$ equals the pull-back
of the $\U(2)$-bundle $f^*V(5)\rightarrow E$  under the map~$\sigma\co
M\rightarrow E$.
\end{lem}

\begin{proof}
Write the two relevant universal bundles as
\[ \pi_{\SO}\co V(5)\rightarrow\BSO(5)\;\;\;\text{and}\;\;\;
\pi_{\U}\co V(5)\rightarrow\BU(2).\]
Then
\[ f^*V(5)=\bigl\{ (x,v)\in M\times V(5)\co f(x)=\pi_{\SO}(v)\bigr\},\]
and the bundle projection $\pi_E\co f^*V(5)\rightarrow E$
is given by
\[ \pi_E(x,v)=\bigl(x,\pi_{\U}(v)\bigr)\in f^*\BU(2)=E.\]
Under $\sigma$ this pulls back to
\[ \bigl\{(x,w)\in M\times f^*V(5)\co \sigma(x)=\pi_E(w)\bigr\},\]
with the obvious projection map to~$M$.
This space can be rewritten as
\[ \bigl\{(x,v)\in M\times V(5)\co \tilde{f}(x)=\pi_{\U}(v)\bigr\},\]
which is the total space of the bundle $\tilde{f}^*V(5)\rightarrow M$.
\end{proof}

The fibre $F_5$ of the bundle $E\rightarrow M$ is diffeomorphic
to $\CP^3$, see~\cite[Lemma~8.1.2 and Proposition~8.1.3]{geig08}.
From the homotopy exact sequence of the generalised Hopf fibration
$S^1\hookrightarrow S^7\rightarrow\CP^3$ one then sees that
the homotopy groups $\pi_i(F_5)$ are trivial for $i=0,1,3,4,5$,
and $\pi_2(F_5)\cong\Z$.

Since the fibre $F_5$ is simply connected, it is in particular
$2$-simple in the sense of \cite[{\S}16.5]{stee51}, i.e.\
the fundamental group operates trivially on $\pi_2(F_5)$.
Moreover, the structure group $\SO(5)$ of the bundle $E\rightarrow M$
is connected. From \cite[{\S}30.4]{stee51} it then
follows that the bundle of coefficients over $M$ whose fibre
over $x$ is the homotopy group $\pi_2(E_x)\cong\Z$ of the fibre~$E_x$ of~$E$
is actually a trivial bundle. This implies that the obstruction
to extending a section of $E$ over the $2$-skeleton of $M$ to
the $3$-skeleton is a cohomology class in $H^3(M;\Z)$. Given two
sections of $E\rightarrow M$ that are homotopic over the $1$-skeleton,
the obstruction to homotopy over the $2$-skeleton lives
in $H^2(M;\Z)$. Similarly, the obstruction cocycles are simply
integral chains.

The obstruction class for the existence of a section over the $3$-skeleton
can be identified with the third integral Stiefel--Whitney class
$W_3(M)$, see~\cite[p.~370]{geig08}. By the vanishing of the other
relevant homotopy groups of~$F_5$, this class is the only obstruction to the
existence of an almost contact structure. Likewise, the only obstruction
to homotopy of two almost contact structures is the primary difference
class in $H^2(M;\Z)$.

We can now formulate the homotopy classification of almost contact
structures. Regarding an almost contact structure as a $\U(2)$-bundle,
we can sensibly speak of its first Chern class~$c_1$, which is a homotopy
invariant. In the following statement and its proof we allow
ourselves to identify an almost contact structure with the
homotopy class it represents. The $k$-skeleton of $M$ will
be denoted by $M^{(k)}$.

\begin{prop}
\label{prop:acs5}
Let $M$ be a compact, oriented $5$-manifold with $W_3(M)=0$.
There is a free and transitive action of $H^2(M;\Z)$ on the
set $\A(M)$ of almost contact structures on~$M$. Write
$u\ast\sigma\in\A(M)$ for the image of $\sigma\in\A(M)$ under the
action of $u\in H^2(M;\Z)$. Then $c_1(u\ast\sigma)=c_1(\sigma)+2u$.
\end{prop}

\begin{proof}
Fix a reference element $\sigma_0\in\A(M)$. For any $u\in
H^2(M;\Z)$, by the first
extension theorem \cite[{\S}37.2]{stee51} of obstruction theory,
we can find a section $\sigma_u'$ of $E$ over $M^{(3)}$ such that
the primary difference class $d(\sigma_0,\sigma_u')\in H^2(M;\Z)$ equals~$u$.
Since the higher relevant homotopy groups of $F_5$ vanish, $\sigma_u'$
can be extended to a section $\sigma_u$ over all of~$M$.

Given any other $\tau_u\in\A(M)$ with primary difference
$d(\sigma_0,\tau_u)= u$, the addition formula \cite[{\S}36.6]{stee51}
implies $d(\sigma_u,\tau_u)=0$, and hence that $\sigma_u$
and $\tau_u$ are homotopic. Thus, $\sigma_u$ denotes a well-defined
homotopy class. This allows us to define a free
and transitive action of $H^2(M;\Z)$
on $\A(M)$ by
\[ u\ast\sigma_v:=\sigma_{u+v}.\]

It remains to prove the formula for the first Chern class.
From the homotopy exact sequence of the universal bundles
we have $\pi_2(\BU(2))\cong\pi_1(\U(2))\cong\Z$ and
$\pi_2(\BSO(5))\cong\pi_1(\SO(5))\cong\Z_2$.
The homotopy exact sequence of the bundle
\[ \CP^3=F_5\hookrightarrow\BU(2)\longrightarrow\BSO(5)\]
then gives us
\[ \begin{array}{ccccccc}
\pi_2(F_5) & \longrightarrow & \pi_2(\BU(2)) & \longrightarrow &
  \pi_2(\BSO(5)) & \longrightarrow & \pi_1(F_5)\\
\Z         & \longrightarrow & \Z            & \longrightarrow &
  \Z_2           & \longrightarrow & 0.
\end{array}\]
It follows that the first homomorphism in this sequence
is multiplication by~$2$.

The inclusion map $\iota\co F_5\rightarrow\BU(2)$ is covered by
a bundle map of $\U(2)$-bundles:
\begin{diagram}
\SO(5) & \rTo         & V(5)\\
\dTo   &              & \dTo\\
F_5    & \rTo^{\iota} & \BU(2)
\end{diagram}
This means that the bundle $\SO(5)\rightarrow F_5$ may be regarded
as the induced bundle $\iota^*V(5)\rightarrow F_5$. By the observation
on the homomorphism $\pi_2(F_5)\rightarrow\pi_2(\BU(2))$, the
homomorphism
\[ \iota^*\co\Z\cong H^2(\BU(2);\Z)\longrightarrow H^2(F_5;\Z)\cong\Z\]
is likewise multiplication by~$2$. Since $H^2(\BU(2);\Z)$ is generated by the
first Chern class, it follows that
$c_1(\iota^*V(5))$ is twice a generator of $H^2(F_5;\Z)\cong\Z$.
Choose a generator of $\pi_2(F_5)= H_2(F_5;\Z)\cong\Z$ and the
corresponding dual generator of $H^2(F_5;\Z)\cong\Z$ --- in other words, fix
an identification of these groups with~$\Z$ --- in such a way that
$c_1(\iota^*V(5))=-2$.

By construction, we have the formula $d(\sigma,u\ast\sigma)=u$
for the difference class. We therefore need to show that
\begin{equation*}
\label{eqn:difference}
\tag{$\ast$}
c_1(\tau)-c_1(\sigma)=2d(\sigma,\tau)\;\;\text{for any $\sigma,\tau\in
\A(M)$}.
\end{equation*}
It suffices to prove this formula over the $2$-skeleton $M^{(2)}$.
Indeed, the inclusion $M^{(2)}\rightarrow M$ induces an
injective homomorphism $H^2(M;\Z)\rightarrow H^2(M^{(2)};\Z)$.

The bundle $E|_{M^{(1)}}$ is trivial; moreover, the fibre
$F_5$ is simply connected. Thus, we may assume that the
sections $\sigma$ and $\tau$ are constant (and identical)
over the $1$-skeleton~$M^{(1)}$.

Recall from \cite[{\S}36]{stee51} the definition of the primary
difference class $d(\sigma,\tau)$, represented by a cochain
with values in $\pi_2(F_5)$.
Any oriented $2$-cell $\Delta\subset M^{(2)}\subset M$
is described by a characteristic map
$\varphi_{\Delta}\co D^2\rightarrow M$ sending $\Int(D^2)$ homeomorphically
onto $\Delta$, and $\partial D^2$ into~$M^{(1)}$.
The section $\sigma$ of the bundle $E\rightarrow M$ defines
a section $\sigma_{\Delta}$ of the pull-back bundle $\varphi_{\Delta}^*E
\rightarrow D^2$ via
\[ \sigma_{\Delta}(x):=\bigl(x,\sigma\circ\varphi_{\Delta}(x)\bigr),\]
likewise for~$\tau$:
\begin{diagram}
D^2\times F_5\cong\varphi_{\Delta}^*E    & \rTo^{\overline{\varphi}_{\Delta}}
     & E\\
\uTo^{\sigma_{\Delta},\tau_{\Delta}}\dTo &
     & \dTo\uTo_{\sigma,\tau}\\
D^2                                      & \rTo_{\varphi_{\Delta}}
     & M\\
\end{diagram}
Notice that the pull-back bundle over $D^2$ is trivial, and in the
trivialisation $\varphi_{\Delta}^*E\cong D^2\times F_5$ the sections
$\sigma_{\Delta},\tau_{\Delta}$ may be regarded as maps $D^2\rightarrow F_5$.
However, there is no \emph{a priori} relation between this
trivialisation and that of $E|_{M^{(1)}}$, so $\sigma_{\Delta},\tau_{\Delta}$
coincide over $\partial D^2$, but they
will not, in general, be constant along~$\partial D^2$.

Write $\pi_{\pm}\co S^2_{\pm}\rightarrow D^2$ for the projection of the upper
and lower hemisphere of the $2$-sphere, respectively, onto the equatorial
disc. Then the class $d(\sigma,\tau)$ is represented by the cocycle
whose value on $\Delta$ is the element of $\pi_2(F_5)$ given by the
map
\[ d(\sigma,\tau)(\Delta)=
\begin{cases}
\sigma_{\Delta}\circ\pi_+ & \text{on $S^2_+$},\\
\tau  _{\Delta}\circ\pi_- & \text{on $S^2_-$}.
\end{cases}\]
Here, by slight abuse of notation, we do not distinguish between cocycles
and the cohomology classes they represent.
The sign convention for the difference class is the standard one as in
\cite[{\S}33.4]{stee51}.

Over $E$ we have the $\U(2)$-bundle $f^*V(5)\rightarrow E$, which
we shall now denote by~$\eta$. Our aim is to compute the difference
$c_1(\tau)-c_1(\sigma)$, which by definition equals
$c_1(\tau^*\eta)-c_1(\sigma^*\eta)$. These Chern classes live
in the cohomology of $M$ with coefficients in the coefficient bundle
$\eta(\pi_1)$ in the notation of \cite[{\S}30.2]{stee51}.
Since the structure group $\U(2)$
has abelian fundamental group $\pi_1(\U(2))\cong \Z$, hence is $1$-simple,
and is connected, again by \cite[{\S}30.4]{stee51} this
coefficient bundle is trivial and we are simply dealing with
integral cohomology classes. (This is of course well known.)

With $\overline{\varphi}_{\Delta}$ defined by the diagram above,
the pull-back bundle $\overline{\varphi}_{\Delta}^*\eta=D^2\times\iota^*V(5)$
restricts to a trivial bundle over either $\sigma_{\Delta}(D^2)$
and $\tau_{\Delta}(D^2)$, and we have sections of these bundles over
$\sigma_{\Delta}(\partial D^2)=\tau_{\Delta}(\partial D^2)$, since
$\sigma\circ\varphi_{\Delta}|_{\partial D^2}=
\tau\circ\varphi_{\Delta}|_{\partial D^2}$ is a constant section of
$E|_{M^{(1)}}\cong M^{(1)}\times F_5$. These sections define elements
of $\pi_1(\U(2))\cong\Z$, and the classes $c_1(\sigma),c_1(\tau)$
are represented by the cochains whose value on $\Delta$ is
precisely that respective element.

It follows that $c_1(\sigma)-c_1(\tau)$ is represented by a cochain
whose value on $\Delta$ is given by the first Chern class of the
$\U(2)$-bundle $\iota^*V(5)$
over the $2$-sphere $d(\sigma,\tau)(\Delta)\in\pi_2(F_5)$.
Since $c_1(\iota^*V(5))=-2$, this implies~($\ast$).
\end{proof}

The following corollary, see~\cite[Theorem~8.18]{hami08},
is then immediate.

\begin{cor}
\label{cor:no2}
In the absence of $2$-torsion in $H^2(M;\Z)$,
almost contact structures are determined up to homotopy by the first
Chern class.
\end{cor}

By completely analogous arguments, one can also prove the
following homotopy classification of almost complex structures on
$6$-manifolds. Again, the third integral Stiefel--Whitney class
is the only obstruction to the existence of an almost complex structure.

\begin{prop}
\label{prop:acs6}
Let $W$ be a compact, oriented $6$-manifold with $W_3(W)=0$.
There is a free and transitive action of $H^2(W;\Z)$ on the
set $\A(W)$ of almost complex structures on~$W$. Write
$u\ast\sigma\in\A(W)$ for the image of $\sigma\in\A(W)$ under the
action of $u\in H^2(W;\Z)$. Then $c_1(u\ast\sigma)=c_1(\sigma)+2u$.
\qed
\end{prop}

\begin{rem}
\label{rem:subcritical}
Let $M$ be a closed, connected $5$-manifold with a subcritical
filling~$W$, i.e.\ a topological filling made up of handles of index
at most two. Dually, $W$ can be obtained from $M$ by attaching
handles of index at least four. The particular consequences
relevant to the discussion below are that the inclusion $M\rightarrow W$
induces isomorphisms both on fundamental groups and on the second
cohomology groups (with any coefficients).
\end{rem}

With this observation, we can formulate a relation between the
sets $\A(M)$ and $\A(W)$. We have a restriction map
$\A(W)\rightarrow\A(M)$ defined by $J\mapsto J|_{TM}$. Here,
by slight abuse of notation, $J\mapsto J|_{TM}$ denotes
the almost contact structure on $M$ given by the coorientable
hyperplane field $TM\cap J(TM)$ with the complex bundle
structure given by~$J$. The isomorphism $H^2(W;\Z)\cong H^2(M;\Z)$
in the following proposition is understood to be the one
induced by the inclusion map $M\rightarrow W$.

\begin{prop}
\label{prop:acs-subcritical}
Let $M$ be a closed, connected $5$-manifold having a subcritical
filling~$W$ with $W_3(W)=0$. Then the restriction map
$\A(W)\rightarrow\A(M)$, $J\mapsto J|_{TM}$ is an
equivariant bijection with respect to the respective actions
of $H^2(W;\Z)\cong H^2(M;\Z)$.
\end{prop}

\begin{proof}
This is immediate from the construction of the action of $H^2(M;\Z)$ on
$\A(M)$ in the proof of Proposition~\ref{prop:acs5},
and the analogous construction for~$W$ to prove Proposition~\ref{prop:acs6},
given that $W$ is obtained from $M$ by attaching handles of index at least
four.
\end{proof}
\section{Topology of subcritically Stein fillable $5$-manifolds}
We first recall the two pertinent results from \cite{ciel12}
in the form in which we need them.

\begin{thm}[{\cite[Theorem~1.5]{ciel12}}]
\label{thm:CE1}
Let $W$ be a compact manifold with boundary, of dimension
$2n\geq 6$, equipped with an almost complex structure~$J$.
If $W$ admits a handle decomposition with handles of index $\leq n$ only,
then $J$ is homotopic to a complex structure $J'$ making $(W,J')$
a Stein domain. The Stein structure can be chosen compatible
with the given handle decomposition.
\end{thm}

The second theorem deals with subcritical Stein domains,
where we have a decomposition into Stein handles of index at most $n-1$.
Notice that the preceding theorem says that if we start
with a subcritical handle decomposition and an almost complex
structure, we can find a subcritical Stein structure.

\begin{thm}[{\cite[Theorem~15.14]{ciel12}}]
\label{thm:CE2}
Let $W$ be a compact manifold with boundary, of dimension
$2n\geq 6$, equipped with almost complex structures $J_1,J_2$
making $(W,J_1)$ and $(W,J_2)$ subcritical Stein domains.
If $J_1$ and $J_2$ are homotopic as almost complex structures, they
are homotopic as Stein structures.
\end{thm}

In fact, Cieliebak--Eliashberg prove this theorem
for so-called \emph{flexible} Stein domains~\cite[Definition~11.29]{ciel12},
which by \cite[Remark~11.30]{ciel12} includes all subcritical ones.

\begin{proof}[Proof of Theorem~\ref{thm:topology}]
Let $(M,\xi)$ be a closed, connected $5$-dimensional contact manifold
with finite cyclic fundamental group, admitting a subcritical
Stein filling $(W,J)$. This Stein filling is made up of handles
of index at most two, so $W$ is simple homotopy equivalent
to a finite $2$-complex, see~\cite[p.~7]{mazu63}.
Dually, as observed before in Remark~\ref{rem:subcritical},
$W$ can be obtained from $M$ by attaching
handles of index at least four.

For $m\geq 2$ an integer, we write $X_m$ for the $m$-fold dunce cap.
This complex is obtained by attaching a $2$-disc $D^2$ to the circle $S^1$
with an attaching map $\partial D^2\rightarrow S^1$ of degree~$m$.
Equivalently, $X_m$ can be obtained as a quotient space of $D^2$ by identifying
each point $x\in S^1=\partial D^2$ with its rotate through
an angle $2\pi/m$.

According to \cite[Theorem~2.1]{hakr93}, given two finite
$2$-complexes $K,K'$ of the same Euler characteristic $\chi(K)=\chi(K')$
and any isomorphism $\pi_1(K)\rightarrow\pi_1(K')$,
where the fundamental group is a finite subgroup of $\SO(3)$,
there is a simple homotopy equivalence $K\rightarrow K'$ inducing the
given isomorphism on fundamental groups.
This implies that if $\pi_1(M)$ is the cyclic group
of order~$m$, the Stein filling $W$
is simple homotopy equivalent to the $2$-complex $X_m\vee_r S^2$,
where $r:=\chi(W)-1$. Thus, $W$ is a thickening
of this $2$-complex in the sense of \cite{mazu63} or~\cite{wall66a}.

Because of $6\geq 2\cdot 2+1$, $6$-dimensional thickenings of
a $2$-complex are in the stable range, and by \cite[Lemma~11.29]{mazu63}
or \cite[Proposition~5.1]{wall66a}, oriented thickenings of
$X_m\vee_r S^2$ are classified up to diffeomorphism by
$[X_m\vee_r S^2,\BSO]$. This set of homotopy classes
is isomorphic to
\[ [X_m\vee_r S^2,K(\Z_2,2)]\cong H^2(X_m\vee_r S^2;\Z_2),\]
since the homotopy groups $\pi_k(\BSO)$ coincide with those
of the Eilenberg--MacLane space $K(\Z_2,2)$ for $k\leq 2$.
The isomorphism
\[ [X_m\vee_r S^2,\BSO]\cong H^2(X_m\vee_r S^2;\Z_2)\]
is given by the second Stiefel--Whitney class, for this obstruction class
detects the non-trivial oriented $\R^{\infty}$-bundle over~$S^2$;
see \cite[Lemma~8.2.5]{geig08}, for instance.

Write $\natural$ for the boundary connected sum of manifolds with boundary,
and $S^2\tilde\times D^4$ for the non-trivial $D^4$-bundle over~$S^2$.
In the case that $m=2n+1$ is odd, we have $H^2(X_{2n+1};\Z_2)=0$,
so a thickening of $X_{2n+1}\vee_r S^2$ is determined by the
tangent bundle over each of the $r$ $2$-spheres being trivial
or not. There is a well-known diffeomorphism
\[ S^2\tilde\times D^4\natural S^2\tilde\times D^4\cong
S^2\tilde\times D^4\natural S^2\times D^4,\]
see \cite[Proposition~4.7]{dgk12}; this diffeomorphism can also
be derived from the argument we shall use presently in the
case that $m$ is even. It follows that $W$ diffeomorphic to
\[ (L_{2n+1}\times D^3)\natural_r(S^2\times D^4)
\;\;\text{or}\;\;
(L_{2n+1}\times D^3)\natural (S^2\tilde\times D^4)\natural_{r-1}
(S^2\times D^4),\]
depending on whether $W$ is spin or not. Since the inclusion
$M\rightarrow W$ induces an isomorphism on $H^2(\,.\,;\Z_2)$,
this proves part (i) of the proposition.

If $m=2n$ is even, we have $H^2(X_{2n};\Z_2)=\Z_2$, so
there is now also a choice of two thickenings over $X_{2n}$.
The same argument as before shows that $W$ is diffeomorphic to
\[ (L_{2n}\times D^3)\natural_r(S^2\times D^4)\]
if $W$ is spin, or, in the non-spin case, to one of the three manifolds
\begin{eqnarray*}
W_{1,0} & := & (L_{2n}\tilde\times D^3)\#_r(S^2\times D^4),\\
W_{0,1} & := & (L_{2n}\times D^3)\#(S^2\tilde\times D^4)
                \#_{r-1}(S^2\times D^4),\\
W_{1,1} & := & (L_{2n}\tilde\times D^3)\#(S^2\tilde\times D^4)
                \#_{r-1}(S^2\times D^4).
\end{eqnarray*}
The manifolds $W_{1,0}$ and $W_{0,1}$ are not diffeomorphic, since
there is no isomorphism $H^2(W_{1,0};\Z)\rightarrow H^2(W_{0,1};\Z)$
of $\Z_{2n}\oplus\Z^r$
whose mod~$2$ reduction sends $w_2(W_{1,0})$ to $w_2(W_{0,1})$;
the same argument applies to the boundaries of these manifolds.

We claim, however, that $W_{1,1}$ is diffeomorphic to $W_{0,1}$;
it suffices to prove this for $r=1$. Both manifolds are obtained
from a $6$-ball by first attaching a $1$-handle to produce
$S^1\times D^5$, and then a couple of $2$-handles $h_1,h_2\cong D^2\times D^4$.
In order to obtain $W_{0,1}$ we attach $h_1$ by an
attaching map $\varphi_1\co\partial D^2\times D^4\rightarrow
\partial(S^1\times D^5)$ that sends $\partial D^2
\times\{0\}$ to a $(2n,-1)$-torus knot on
\[ S^1\times\partial D^2\times\{0\}\subset S^1\times\partial D^2\times D^3
\subset\partial(S^1\times D^5),\]
extended to an embedding of $\partial D^2\times D^4$ with trivial framing.
The handle $h_2$ is attached by a map $\varphi_2$
sending $\partial D^2\times\{0\}$ to a homotopically trivial
circle in
\[ S^1\times\partial D^2\times\{0\}\setminus\varphi_1\bigl(\partial D^2
\times D^4\bigr),\]
extended to an embedding of $\partial D^2\times D^4$ with twisted
framing, corresponding to the non-trivial element of $\pi_1(\SO(4))=\Z_2$.

By sliding $h_1$ over $h_2$ we get a diffeomorphic manifold where
the framing of the first handle is now also twisted, in other words,
the manifold~$W_{1,1}$. This proves part (ii) of the proposition.

Finally, we come to the statement about the existence of
subcritically fillable contact structures. Let $M$ be one of the
$5$-manifolds in (i) or (ii), and $W$ the corresponding $6$-manifold discussed
in the course of proving this classification, with boundary~$M$.
This manifold $W$ admits an almost complex structure, since
$H^3(W;\Z)=0$.
By Theorem~\ref{thm:CE1}, any almost complex structures on
$W$ is homotopic to a subcritical Stein structure.
From Proposition~\ref{prop:acs-subcritical} it follows that
every homotopy class of almost contact structures on $M$
contains a subcritically Stein fillable contact structure.
According to Theorem~\ref{thm:CE2}, the subcritically Stein
fillable contact structure within a given homotopy class is unique
up to isotopy, and its Stein filling is unique up to Stein homotopy.
\end{proof}

\begin{rem}
(1) Let $M$ be one of the $5$-manifolds in Theorem~\ref{thm:topology}
and $\xi$ a subcritically Stein fillable contact structure as
just described. Then, by \cite[Theorem~5.3]{bgz16}, any
symplectically aspherical filling is homotopy equivalent to
the corresponding~$W$, and even diffeomorphic to $W$ if the
Whitehead group of $\Z_m$ vanishes, which by
\cite[Corollary~6.5]{miln66} happens exactly
for $m\in\{2,3,4,6\}$.

(2) From \cite[Theorem~B]{hakr93} one can derive the
following stabilisation result. Let $(M,\xi)$ be a closed, connected
contact $5$-manifold with finite
fundamental group, admitting a subcritical
Stein filling $(W_0,J_0)$.
Then any subcritical Stein filling
of any contact structure on $M\# S^2\times S^3$ is simple
homotopy equivalent to $W_0\natural S^2\times D^4$. Homotopically,
the additional summand amounts to a one-point union with $S^2$.
\end{rem}
\section{Uniqueness of subcritically Stein fillable $5$-manifolds}
\begin{proof}[Proof of Theorem~\ref{thm:pi1=1}]
By the proof of \cite[Proposition~7.4]{bcs14}, the filling $W$
is diffeomorphic to $\natural_r(S^2\times D^4)$ or
$(S^2\tilde{\times}D^4)\natural_r(S^2\times D^4)$, where
$r$ is the same non-negative integer as in the description of~$M$; this also
follows from \cite[Theorem~1.5]{bgz16}.

The theorem then follows by the same argument as the
one we used at the end of the proof of Theorem~\ref{thm:topology}.
\end{proof}

\begin{proof}[Proof of Corollary~\ref{cor:pi1=1}]
(a) This is an immediate consequence of Theorem~\ref{thm:pi1=1}.
Alternatively, here is a more direct proof.
Suppose that $\xi$ is a contact structure on $S^5$ that admits
a subcritical Stein filling $(W,J)$. This means that there
is a contactomorphism $f\co (S^5,\xi)\rightarrow(\partial W,\xi_J)$,
where $\xi_J$ denotes the contact structure induced by~$J$.

By \cite[Theorem~1.2]{bgz16} or an earlier result of
Oancea--Viterbo~\cite{oavi12}, see the discussion in
\cite[Section~3.3]{bgz16}, the manifold $W$ is a simply
connected homology ball, and hence diffeomorphic to the
standard ball $D^6$ by Proposition~A on page 108 of~\cite{miln65}.
Choose an orientation-preserving diffeomorphism $G\co D^6\rightarrow W$.

Recall that any diffeomorphism of $S^5$ can be extended to
a diffeomorphism of $D^6$; this corresponds with the fact that
there are no exotic $6$-spheres, see \cite{kemi63} and
\cite[Corollary VIII.(5.6)]{kosi93}.
Let $\Phi\co D^6\rightarrow D^6$ be an orientation-preserving diffeomorphism
extending the diffeomorphism $\bigl(G|_{S^5}\bigr)^{-1}\circ f\co
S^5\rightarrow S^5$. Notice that the diffeomorphism
$F:=G\circ\Phi\co D^6\rightarrow W$ restricts to $f$ on~$S^5$.

It follows that the subcritical Stein structure $F^*J$ on $D^6$ induces
the contact structure $\xi$ on~$S^5$. Write $\Jst$ for the
standard complex structure on $D^6$ inducing the standard contact structure
$\xist$ on~$S^5$. The two complex structures $F^*J$ and $\Jst$
on $D^6$ are homotopic as almost complex structures. Thus,
by Theorem~\ref{thm:CE2}, the respective induced contact
structures $\xi$ and $\xist$ on $S^5$ are isotopic.

(b) One way to prove this is by appealing to the
results of Barden~\cite{bard65} on the classification and
the diffeomorphisms of simply connected $5$-manifolds,
as nicely expounded in \cite[Chapter~VII]{hami08}, see
in particular \cite[Theorem~7.16]{hami08}.
Under the assumptions of the corollary (and given the fact
from \cite{bcs14} that simply connected $5$-manifolds
admitting a subcritical Stein filling necessarily have
torsion-free homology), there is a diffeomorphism
$M_2\rightarrow M_1$ that induces the given isomorphism on~$H^2$.
Then argue as in the proof of Theorem~\ref{thm:pi1=1}.
\end{proof}

Here is an alternative argument for part (b) of the
corollary that avoids having to cite
the result of Barden on the diffeomorphisms of
simply connected $5$-manifolds. This argument is,
in some sense, more constructive, since it reduces the problem
to the diagrammatic language of~\cite{dgk12}.

We want to show that any closed, simply connected
contact $5$-manifold $(M,\xi)$ that admits a subcritical Stein
filling $(W,J)$ also admits a subcritical Stein filling without $1$-handles.
Then the theorem follows from the corresponding result
\cite[Theorem~4.8]{dgk12}, which made precisely this additional
assumption on the absence of $1$-handles.

Again we use the fact (as in the proof of Theorem~\ref{thm:pi1=1})
that for a given $M$ the topology of the filling $W$ is
known. As shown in the
proof of \cite[Proposition~4.5]{dgk12}, for any class $c\in H^2(W;\Z)$
that reduces modulo~$2$ to the Stiefel--Whitney class $w_2(W)$,
there is a subcritical Stein structure on $W$ 
\emph{without $1$-handles} with first
Chern class~$c$. (Moreover, the cited proposition shows directly that
the contact structure induced on the boundary is determined by~$c$.)

In particular, we find such a subcritical Stein structure $J'$
with $c_1(J')=c_1(J)$. Since $W$ is simply connected,
the analogue of Corollary~\ref{cor:no2} shows that
$J$ and $J'$ are homotopic as almost complex structures.
By Theorem~\ref{thm:CE2}, this implies
that $J$ and $J'$ are actually Stein homotopic.
Thus, as claimed, the stipulation that there be no $1$-handles poses no 
restriction.

\begin{ack}
F.~D.\ would like to thank Yakov Eliashberg and Otto van Koert
for helpful conversations. We thank the referee for
suggesting some improvements to the content and the
exposition of this paper. F.~D.\ and G.~Z.\ are supported by grant
no.\ 11371033 of the National Natural Science Foundation of China.
H.~G.\ is supported by the SFB/TRR 191 `Symplectic Structures
in Geometry, Algebra and Dynamics', funded by the Deutsche
Forschungsgemeinschaft.
\end{ack}

\end{document}